\documentclass[11pt]{article}

\usepackage{amsmath, amssymb, amsthm, calc}  

\title{Linear forms of a given Diophantine type.
                              \thanks{ This research was  supported by
                              RFBR (grant $\textup N^{\circ}$ 09--01--00371a).
                              The first author was also supported by
                              the grant of the President of Russian Federation
                              $\textup N^\circ$ MK--4466.2008.1.
                             }}
\author{Oleg\,N.\,German, Nikolay\,G.\,Moshchevitin}
\date{}

\theoremstyle{definition}
\newtheorem{definition}{Definition}

\theoremstyle{remark}
\newtheorem{remark}{Remark}

\theoremstyle{plain}
\newtheorem{theorem}{Theorem}
\newtheorem{lemma}{Lemma}

\newtheorem{corollary}{Corollary}

\DeclareMathOperator{\spanned}{span}
\DeclareMathOperator{\interior}{int}

\renewcommand{\vec}[1]{\mathbf{#1}}
\renewcommand{\geq}{\geqslant}
\renewcommand{\leq}{\leqslant}

\newcommand{\e}{\varepsilon}
\newcommand{\R}{\mathbb{R}}
\newcommand{\Z}{\mathbb{Z}}
\newcommand{\Q}{\mathbb{Q}}

\begin{document}

  \maketitle

  \begin{abstract}
    We prove a result on the existence
    of  linear forms of a given Diophantine type.
  \end{abstract}

  \section{Approximation to irrational numbers} \label{sec:approx_of_irrationals}

  Let $\alpha$ be an irrational number. The Hurwitz theorem says that the inequality
  \[ \|q\alpha \| <\frac{1}{\sqrt{5} q}, \]
  where $\|\cdot\|$ denotes the distance to the nearest integer, has infinitely many solutions in
  integer $q$. Moreover, there is a countable set of numbers $\alpha$ for which this inequality
  is exact, that is for every positive $\e$ the inequality
  \[ \|q\alpha\|<\left(\frac{1}{\sqrt{5}}-\e\right)\frac{1}{ q} \]
  admits only a finite number of solutions in integer $q$.

  The numbers $\lambda$ under the condition that there is an $\alpha=\alpha(\lambda)$ for which
one has
  \[ \lambda = \liminf_{q\to+\infty} q||q\alpha || \]
  form {\it the Lagrange spectrum}. It is a well-known
  fact that the Lagrange  spectrum has a discrete part
  \[ \frac{1}{\sqrt{5}},\frac{1}{\sqrt{8}},\ldots, \]
  and the maximal $\lambda$ for which there are continuously many $\alpha (\lambda )$ is
  $\lambda=1/3$. It is also well-known that the Lagrange spectrum contains an interval
  $[0;\lambda^*]$. This interval is known as {\it Hall's ray} as M. Hall \cite{HALL} was the first
  to prove that $\lambda^*>0$. These and many other results concerning the Lagrange spectrum can be
  found in the book \cite{CusM}.

  Furthermore, V. Jarnik (see \cite{Jarnik1}, Satz 6) showed that for every decreasing function
  $\varphi(y)=o(y^{-1})$ there is an uncountable set of real numbers $\alpha$ satisfying the following
  conditions: the inequality
  \[ \|q\alpha\|<\varphi(q) \]
  has infinitely many solutions but for any $\e >0$ the stronger inequality
  \[ \|q\alpha\|<(1-\e)\varphi(q) \]
  has only a finite number of solutions.

  The results mentioned above use the theory of continued fractions.

  Recently V. Beresnevich, H. Dickinson and S. Velani in \cite{BDV} and Y. Bugeaud in \cite{BU},
  \cite{B1} obtained a precise  metric version of Jarnik's result. For example, Y. Bugeaud
  \cite{BU} showed that for every decreasing function $\varphi:\R_+\to\R_+$, such that the function
  $x\mapsto x^2\varphi(x)$ is non-increasing and the series
  \begin{equation} \label{seri}
    \sum_{x=1}^\infty x\varphi(x)
  \end{equation}
  converges, the sets
  \[ \mathcal K(\varphi)=\left\{\alpha\in\R:\ \left|\alpha-\frac{p}{q}\right|<\varphi(q)\
     \text{ for infinitely many rationals }\ \frac{p}{q}\right\} \]
  and
  \[ \textup{EXACT}(\varphi)=\mathcal K(\varphi)\setminus
     \left(\bigcup_{\e>0}\mathcal K((1-\e)\varphi)\right) \]
  have the same Hausdorff dimension. Moreover, Y. Bugeaud \cite{BU} proved that the sets
  $\mathcal K(\varphi)$ and $\textup{EXACT}(\varphi)$ have the same $\mathcal H^f$-Hausdorff measure for a
  certain choice of the dimension function $f:\R_+\to\R_+$. A certain result in the case when the
  series \eqref{seri} diverges was obtained by Y. Bugeaud in \cite{B1}.

  \section{General result by Jarnik} \label{sec:Jarnik}

  Throughout the paper for each $\vec x=(x_1,\ldots,x_n)$ we denote by $|\vec x|$ the Euclidean
  norm
  \[ |\vec x|=|\vec x|_2=(x_1^2+\ldots+x_n^2)^{1/2} \]
  and by $|\vec x|_\infty$ the sup-norm
  \[ |\vec x|_\infty=\max_{1\leq i\leq n}|x_i|. \]
  We also denote by $\langle\pmb\alpha,\pmb\beta\rangle$ the inner product of
  $\pmb\alpha,\pmb\beta\in\R^n$. For a fixed $\pmb\alpha$ we get a linear form
  $\langle\pmb\alpha,\cdot\rangle$.

  We formulate a general result from \cite{Jarnik2}. Consider a real matrix
  \[ \Theta = \left( \begin{array}{ccc}
                       \theta_{1,1} & \cdots & \theta_{1,n} \\
                       \cdots & \cdots & \cdots \\
                       \theta_{m,1} & \cdots & \theta_{m,n}
                     \end{array}\right). \]
  Given a function $\varphi:\R_+\to\R_+$ we say that a set of $n+m$ integers
  \[ x_1,\ldots,x_n,\,y_1,\ldots,y_m \]
  is a \emph{$\varphi$-approximation for $\Theta$} if with
  $\pmb\theta_i=(\theta_{i,1},\ldots,\theta_{i,n})$ we have
  \[ \begin{cases}
       |\langle\pmb\theta_i,\vec x\rangle-y_i|<\varphi(|\vec x|_\infty),\quad i=1,\ldots m, \\
       |\vec x|_\infty>0.
     \end{cases} \]
  V. Jarnik considered an arbitrary non-increasing function $\varphi:\R_+\to\R_+$ and an arbitrary
  function $\lambda:\R_+\to\R_+$, such that the following conditions are satisfied:

  \begin{list}{$\bullet$}{}
    \item $\lambda(x)\to0,\quad x\to\infty$; \vphantom{$\displaystyle\int$}
    \item the functions $\varphi(x)\cdot x^{1/k}$,\ \ $k=1,\ldots,m$,\ \
        $\varphi(x)\cdot x^{1+\e}$ and $\varphi(x)\cdot x^{(n-1)/m}$ are monotone;
    \item the integral $\displaystyle\int_A^\infty x^{n-1}(\varphi(x))^mdx$ converges.
  \end{list}
  For such $\varphi(x)$ and $\lambda(x)$ he proved in \cite{Jarnik2}\footnote{for the special case
  of simultaneous approximations ($n=1$) see \cite{Jarnik1}} that there is an uncountable set of
  matrices $\Theta$, each having infinitely many $\varphi$-approximations but not more than a
  finite collection of $\lambda\varphi$-approximations.

  Another result by Jarnik (see \cite{Jarnik2}, Th\'eor\`eme B) gives a more precise statement under
  stronger conditions on $\varphi(x)$. Namely, he considered an arbitrary function $\varphi(x)$
  satisfying the following conditions:
  \begin{list}{$\bullet$}{}
    \item $\varphi(x)\cdot x\to0,\quad x\to\infty$; \vphantom{$\displaystyle\int$}
    \item the functions $\varphi(x)\cdot x$ and $\varphi(x)\cdot x^{(n-1)/m}$ are monotone;
    \item the integral $\displaystyle\int_A^\infty x^{n-1}(\varphi(x))^mdx$ converges;
  \end{list}
  and proved the existence of an uncountable set of matrices $\Theta$, each having infinitely many
  $\varphi$-approximations but not more than a finite collection of $(1-\e)\varphi$-approximations,
  for any positive $\e$.

  So we see that the additional condition
  \[ \varphi(x)=o(x^{-1}),\quad x\to\infty, \]
  allows obtaining sharper results concerning systems of linear forms of a given Diophantine type.

  Here we would like to note that a nice metric generalization of Jarnik's result was obtained by
  V. Beresnevich, H. Dickinson and S. Velani in \cite{BDV}. There the authors deal with a general
  setting for Diophantine approximations for systems of linear forms and prove certain results on
  the ``exact logarithmic'' order of approximations.

  In the next two  sections we discuss some improvements of Jarnik's result in the cases
  $n=1$, $m\geq2$ (simultaneous approximations) and $m=1$, $n\geq2$ (linear forms). Here we should
  note that there is an old problem, still unsolved, to generalize  the result on the existence
  of Hall's ray mentioned in Section \ref{sec:approx_of_irrationals} to the cases of simultaneous
  approximations and linear forms (and even to the general case). This problem seems to be a
  difficult one.

  \section{Simultaneous approximations}

In this section for convenience we put
\[
\varphi(t)  =\frac{\psi (t)}{t^{1/m}}.\]

  In Jarnik's theorem discussed in Section \ref{sec:Jarnik} by certain reasons the monotonicity
  conditions may be omitted. This observation in the case of simultaneous approximations (see also
  \cite{Jarnik1}, Satz 5) leads to the following

  \begin{theorem}
    Let $m$ be a positive integer.
    Given an arbitrary decreasing function $\psi:\R_+\to\R_+$ and an arbitrary function
    $\lambda:\R_+\to\R_+$, decreasing to zero, suppose that the integral
    \[ \int_A^\infty \frac{(\psi(x))^m}{x}dx \]
    converges. Then one can find an uncountable set of
    $\pmb\alpha=(\alpha_1,\dots,\alpha_m)\in\R^m$, such that for every sufficiently large positive integer $q$ one has
    \[ \max_{1\leq i\leq m}\|q\alpha_i\|\geq\frac{\lambda (q)\psi(q)}{q^{1/m}}, \]
    but the inequality
    \[ \max\limits_{1\leq i\leq m}\|q\alpha_i\|\leq\frac{\psi(q)}{q^{1/m}} \]
    has infinitely many solutions in positive integer $q$.
  \end{theorem}

  In \cite{MAH} R. Akhunzhanov and  N. Moshchevitin  generalizing the approach from \cite{MOSS}
  proved the following

  \begin{theorem}
    Let $m $ be a positive integer. Then there are explicit positive constants $A_m$, $B_m$ with the
    following property. Given an arbitrary non-increasing function $\psi:\R_+\to\R_+$,
    $\psi(1)\leq A_m$, one can find an uncountable set of vectors
    $\pmb\alpha=(\alpha_1,\dots,\alpha_m)\in\R^m$ such that for every positive integer $q$
    \[ \max_{1\leq i\leq m}\|q\alpha_i\|{\geq} \frac{\psi(q)}{q^{1/m}}\left(1-B_m\psi(q)\right), \]
    but the inequality
    \[ \max_{1\leq i\leq m}\|q\alpha_i\|{\leq}
       \frac{\psi(q)}{q^{1/m}}\left(1{+}B_m\psi(q)\right) \]
    has infinitely many solutions in positive integer $q$.
  \end{theorem}

  Here we would like to note that in the paper \cite{MOSS} the author attributes to V. Jarnik a
  stronger result than he actually proved in \cite{Jarnik1}, \cite{Jarnik2}.

  \section{Linear forms}

In this section we put
\[
\varphi(t) =\frac{\psi(t)}{t^n}.\]
  Jarnik's theorem discussed in  Section \ref{sec:Jarnik} in the case of  linear forms leads to the
  following

  \begin{theorem}
    Let $n$ be a positive integer.
    Given an arbitrary decreasing function $\psi:\R_+\to\R_+$ and an arbitrary function
    $\lambda:\R_+\to\R_+$, decreasing to zero, suppose that the integral
    \[ \int_A^\infty \frac{\psi(x)}{x}dx \]
    converges. Then there is an uncountable set of $\pmb\alpha\in\R^n$, such that for all
    $\vec x\in\Z^n\setminus\{\vec 0\}$
with $|\vec x |$ sufficiently large one has
    \[ \|\langle\pmb\alpha,\vec x\rangle\|\geq
       \frac{\lambda(\vec x)\psi(\vec x)}{|\vec x|_\infty^n}, \]
    but the inequality
    \[ \|\langle\pmb\alpha,\vec x\rangle\|\leq\frac{\psi(\vec x)}{|\vec x|_\infty^n} \]
    has infinitely many solutions in $\vec x\in\Z^n$.
  \end{theorem}

  We now formulate the main result of this paper. For simplicity we restrict ourselves to the
  case $n=2$ and use the Euclidean norm, though we believe that a similar result should be valid
  for systems of linear forms and for arbitrary norms.

  \begin{theorem} \label{t:intro_main}
    There are explicit positive constants $A$, $B$ with the following property. Given an arbitrary
    non-increasing function $\psi:\R_+\to\R_+$, $\psi(1)\leq A$, one can find an uncountable
    set of $\pmb\alpha\in\R^2$, such that for all $\vec x\in\Z^2\setminus\{\vec 0\}$
    \[ \|\langle\pmb\alpha,\vec x\rangle\|\geq
       \frac{\psi(|\vec x|)}{|\vec x|^2}(1-B\psi(|\vec x|)) \]
    but the inequality
    \[ \|\langle\pmb\alpha,\vec x\rangle\|\leq
       \frac{\psi(|\vec x|)}{|\vec x|^2}(1+B\psi(|\vec x|)) \]
    has infinitely many solutions in $\vec x\in\Z^2$.
  \end{theorem}

  \section{Best approximations}

  \begin{definition}
    A point $\vec m\in\Z^2\backslash\{\vec 0\}$ is said to be a \emph{best
    approximation} for $\langle\pmb\alpha,\cdot\rangle$ if
    \[ \|\langle\pmb\alpha,\vec m\rangle\|<\|\langle\pmb\alpha,\vec m'\rangle\| \]
    for every $\vec m'\in\Z^2\backslash\{\vec 0\}$, such that $|\vec m'|<|\vec m|$, and
    \[ \|\langle\pmb\alpha,\vec m\rangle\|\leq\|\langle\pmb\alpha,\vec m'\rangle\| \]
    for every $\vec m'\in\Z^2\backslash\{\pm\vec m\}$, such that $|\vec m'|=|\vec m|$.
  \end{definition}

  The set of all the best approximations for $\langle\pmb\alpha,\cdot\rangle$ is infinite if and
  only if the coordinates of $\pmb\alpha$ are linearly independent with the unit over $\Q$. If this
  is the case, then for each possible absolute value there are exactly two best approximations, on
  which this absolute value is attained, and they differ only in the sign. Thus, we can order the
  set of all the best approximations for $\langle\pmb\alpha,\cdot\rangle$ with respect to the
  absolute value and obtain a sequence $\{\pm\vec m_k\}_{k=1}^\infty$.

  Set
  \begin{equation} \label{eq:gamma_0}
    \gamma=\frac{18}{9-\sqrt2}\approx2.373\,.
  \end{equation}

  Theorem \ref{t:intro_main} is a corollary of the following, more precise, theorem, which is the
  main result of the paper. In this theorem the \emph{whole} sequence of best approximations is
  concerned and \emph{all} of them are required to be of a given order.

  \begin{theorem} \label{t:main}
    Given an arbitrary non-increasing function $\psi:\R_+\to\R_+$, $\psi(1)\leq(9\gamma)^{-1}$,
    there is an $\pmb\alpha\in\R^2$, such that all the best approximations $\vec m_k$ for
    $\langle\pmb\alpha,\cdot\rangle$ satisfy the condition
    \begin{equation} \label{eq:theorem_inequality_for_all_the_BA}
      \psi(|\vec m_k|)-4\gamma\psi(|\vec m_k|)^2<
      \|\langle\pmb\alpha,\vec m_k\rangle\|\cdot|\vec m_k|^2\leq
      \psi(|\vec m_k|)+\gamma\psi(|\vec m_k|)^2.
    \end{equation}
    Moreover, there is a continuum of such $\pmb\alpha$.
  \end{theorem}

  We note that the technique used here to prove Theorem \ref{t:main} is similar to the technique
  developed in \cite{german_asymptotic_directions}.

  \section{Proof of Theorem \ref{t:main}}

  \subsection{Description of the set of forms having a given point $\vec m$ as a best approximation}

  Given a primitive point $\vec m\in\Z^2$, the set of all $\pmb\alpha\in\R^2$, such that $\vec m$
  is a best approximation for $\langle\pmb\alpha,\cdot\rangle$, is contained in the set
  \begin{equation} \label{eq:contained_in}
    \mathfrak S=
    \bigcap_{\substack{\vec m'\in\Z^2\backslash\{\vec 0,\pm\vec m\} \\ |\vec m'|\leq|\vec m|}}
    \Big\{ \vec x\in\R^2\, \Big|\, \|\langle\vec m,\vec x\rangle\|\leq\|\langle\vec m',\vec x\rangle\| \Big\}
  \end{equation}
  and contains its interior. One can easily see that each of the sets in the intersection
  \eqref{eq:contained_in} is simply a union of parallelograms. Besides that, one of the two
  diagonals of each of these parallelograms lies on an integer level of the form $\langle\vec m,\cdot\rangle$, and
  the union of such diagonals coincides with the union of all the integer levels of $\langle\vec m,\cdot\rangle$.
  Thus, all the connected components of $\interior\mathfrak S$ are open convex polygons. None of
  these polygons can be too small. To see this we shall use the following

  \begin{lemma} \label{l:distance_divisibility}
    Let $\vec a,\vec b,\vec c\in\Z^2$ and let $\vec b$ and $\vec c$ be linearly independent. Let
    also $\pmb\alpha\in\R^2$, $\langle\vec b,\pmb\alpha\rangle\in\Z$,
    $\langle\vec c,\pmb\alpha\rangle\in\Z$. Then for every $\lambda\in\Z$ the (Euclidean) distance
    from $\pmb\alpha$ to the line defined by the equation $\langle\vec a,\vec x\rangle=\lambda$
    is an integer multiple of
    \[ \frac{1}{|\vec a||\det(\vec b,\vec c)|}\,. \]
  \end{lemma}

  \begin{proof}
    The index of $\Z^2$ as a sublattice of the lattice, dual to $\spanned_\Z(\vec b,\vec c)$,
    is equal to $|\det(\vec b,\vec c)|$. Hence $\det(\vec b,\vec c)\pmb\alpha\in\Z^2$, and thus,
    \[ \langle\vec a,\pmb\alpha\rangle\in\frac{\Z}{|\det(\vec b,\vec c)|}\,. \]
    It remains to notice that the Euclidean distance between two adjacent integer levels of the
    form $\langle\vec a,\vec x\rangle$ equals $1/|\vec a|$.
  \end{proof}

  Since not all the linear forms determining the boundary of a connected component of
  $\interior\mathfrak S$ are necessarily integer, but some of them may have half--integer
  coefficients, the fact that none of those components can be too small is implied by the following
  obvious corollary to Lemma \ref{l:distance_divisibility}:

  \begin{corollary} \label{cor:distance_divisibility}
    Let $\vec a\in\frac{1}{2}\Z^2$, $\vec b,\vec c\in\Z^2$ and let $\vec b$ and $\vec c$ be
    linearly independent. Let also $\pmb\alpha\in\R^2$, $\langle\vec b,\pmb\alpha\rangle\in\Z$,
    $\langle\vec c,\pmb\alpha\rangle\in\Z$. Then for every $\lambda\in\frac{1}{2}\Z$ the (Euclidean)
    distance from $\pmb\alpha$ to the line defined by the equation
    $\langle\vec a,\vec x\rangle=\lambda$ is an integer multiple of
    \[ \frac{1}{2|\vec a||\det(\vec b,\vec c)|}\,. \]
  \end{corollary}

  \subsection{Basis change}

  The following statement supports one of the crucial steps in the proof of Theorem \ref{t:main}.

  \begin{lemma} \label{l:triple_switch}
    Suppose $\vec a,\vec b,\vec c\in\Z^2$ and $\pmb\alpha,\pmb\beta\in\R^2$ satisfy the relations
    $\langle\pmb\alpha,\vec b\rangle\in\Z$,
    $\langle\pmb\alpha,\vec c\rangle=\langle\pmb\beta,\vec c\rangle\in\Z$, and
    $\langle\pmb\beta,\vec a\rangle$ equals the nearest integer to
    $\langle\pmb\alpha,\vec a\rangle$ (in case $\langle\pmb\alpha,\vec a\rangle=1/2$ one can take
    any of the two nearest integers). Suppose also that $\vec b,\vec c$ are linearly independent,
    $\vec a,\vec c$ are linearly independent,
    \begin{equation} \label{eq:projection_of_supplementability}
      \big\{ \vec x\in\Z^2\, \big|\, \langle\pmb\alpha,\vec x\rangle\in\Z \big\}=
      \spanned_\Z(\vec b,\vec c)
    \end{equation}
    and
    \begin{equation} \label{eq:a_is_the_closest}
      \|\langle\pmb\alpha,\vec a\rangle\|=|\det(\vec b,\vec c)|^{-1}.
    \end{equation}
    Then
    \begin{equation} \label{eq:projection_of_supplementability_switched}
      \big\{ \vec x\in\Z^2\, \big|\, \langle\pmb\beta,\vec x\rangle\in\Z \big\}=
      \spanned_\Z(\vec a,\vec c)
    \end{equation}
    and
    \begin{equation} \label{eq:b_is_the_closest}
      \|\langle\pmb\beta,\vec b\rangle\|=|\det(\vec a,\vec c)|^{-1}.
    \end{equation}
  \end{lemma}

  \begin{proof}
    Let $\vec a=(a_1,a_2)$, $\vec b=(b_1,b_2)$, $\vec c=(c_1,c_2)$ and set
    $a_3=-\langle\pmb\beta,\vec a\rangle$, $b_3=-\langle\pmb\alpha,\vec b\rangle$,
    $c_3=-\langle\pmb\alpha,\vec c\rangle=-\langle\pmb\beta,\vec c\rangle$.

    Let us prove that the points
    \[ \overline{\vec a}=(a_1,a_2,a_3),\quad\overline{\vec b}=(b_1,b_2,b_3),\quad
       \overline{\vec c}=(c_1,c_2,c_3) \]
    form a basis of $\Z^3$.
    It follows from \eqref{eq:projection_of_supplementability} that all the integer points
    contained in the plane $\pi_\alpha$ spanned by $\overline{\vec b}$ and $\overline{\vec c}$
    belong to the lattice
    \[ \spanned_\Z(\overline{\vec b},\overline{\vec c}). \]
    This means that $\overline{\vec b}$ and $\overline{\vec c}$ can be supplemented to a
    basis of $\Z^3$. Hence $\Z^3$ splits into ``layers'' contained in two-dimensional planes
    parallel to $\pi_\alpha$. Moreover, any two neighbouring planes cut a segment in the vertical
    axis of length $|\det(\vec b,\vec c)|^{-1}$. Now, using \eqref{eq:a_is_the_closest} and the
    fact that $a_3$ equals the nearest integer to $-\langle\pmb\alpha,\vec a\rangle$ we see that
    $\overline{\vec a}$ lies in a plane next to $\pi_\alpha$. This shows that
    $\overline{\vec a},\overline{\vec b},\overline{\vec c}$ form a basis of $\Z^3$.

    Thus, all the integer points of the plane $\pi_\beta$ spanned by $\overline{\vec a}$ and
    $\overline{\vec c}$ are in
    \[ \spanned_\Z(\overline{\vec a},\overline{\vec c}), \]
    which immediately implies \eqref{eq:projection_of_supplementability_switched}. As before,
    $\Z^3$ splits into ``layers'' contained in two-dimensional planes parallel to $\pi_\beta$, such
    that any two neighbouring ones cut a segment in the vertical axis of length
    $|\det(\vec a,\vec c)|^{-1}$. Noticing that $\overline{\vec b}$ lies in a plane next to
    $\pi_\beta$ we get \eqref{eq:b_is_the_closest}.
  \end{proof}

  \subsection{Induction lemma}

  To prove theorem \ref{t:main} we shall construct a sequence of
  embedded two-dimensional ``half-balls'' $\{\Omega_k\}_{k=1}^\infty$ with their common point
  $\pmb\alpha$ satisfying the statement of the corresponding theorem. We say that a set $\Omega$ is
  a \emph{half-ball of radius $R$ centered at a point $\vec x$} if $\Omega$ is the intersection of
  a closed Euclidean ball of radius $R$ centered at $\vec x$ and a closed half-plane with the
  supporting line containing $\vec x$.

  The following lemma gives the induction step.

  \begin{lemma} \label{l:induction_step}
    Let $k\in\Z_+$, $\psi_k,\psi_{k+1}\in\R_+$, $\psi_k,\psi_{k+1}\leq(9\gamma)^{-1}$, and let
    $\vec m_k,\vec m_{k+1}\in\Z^2$. Let $\Omega_k\subset\R^2$ be a half-ball of radius
    \[ R_k=(2|\vec m_{k+1}||\det(\vec m_k,\vec m_{k+1})|)^{-1} \]
    centred at $\pmb\alpha_k$ with the line
    \[ \ell_k=\big\{ \vec x\in\R^2\, \big|\,
       \langle\vec x,\vec m_k\rangle=\langle\pmb\alpha_k,\vec m_k\rangle \big\} \]
    supporting it. Suppose that the following conditions are satisfied:

    $1)$ $\langle\vec m_k,\vec m_{k+1}\rangle\leq0$;

    $2)$ $\big\{ \vec x\in\Z^2\, \big|\, \langle\pmb\alpha_k,\vec x\rangle\in\Z \big\}=
    \spanned_\Z(\vec m_k,\vec m_{k+1})$;

    $3)$ for every $\pmb\alpha\in\Omega_k$ and every
    $\vec m\in\Z^2\backslash\spanned_\Z(\vec m_k,\vec m_{k+1})$, such that
    $|\vec m|<|\vec m_{k+1}|$, one has
    \[ \|\langle\pmb\alpha,\vec m_k\rangle\|<\|\langle\pmb\alpha,\vec m\rangle\|; \]

    $4)$ $\gamma<\dfrac{|\det(\vec m_k,\vec m_{k+1})|}{|\vec m_k|^2}<3\gamma$;

    $5)$ $(2\gamma\psi_k)^{-1/2}\leq\dfrac{|\vec m_{k+1}|}{|\vec m_k|}<(\gamma\psi_k)^{-1/2}$.

    Then there is a point $\vec m_{k+2}\in\Z^2$, linearly independent with $\vec m_{k+1}$, and a
    half-ball $\Omega_{k+1}\subset\Omega_k$ of radius
    \[ R_{k+1}=(2|\vec m_{k+2}||\det(\vec m_{k+1},\vec m_{k+2})|)^{-1} \]
    centred at $\pmb\alpha_{k+1}$ with the line
    \[ \ell_{k+1}=\big\{ \vec x\in\R^2\, \big|\,
       \langle\vec x,\vec m_{k+1}\rangle=\langle\pmb\alpha_{k+1},\vec m_{k+1}\rangle \big\} \]
    supporting it which satisfy the following conditions:

    $1)$ $\langle\vec m_{k+1},\vec m_{k+2}\rangle\leq0$;

    $2)$ $\langle\pmb\alpha_{k+1},\vec m_{k+1}\rangle=\langle\pmb\alpha_k,\vec m_{k+1}\rangle$ and
    $\big\{ \vec x\in\Z^2\, \big|\, \langle\pmb\alpha_{k+1},\vec x\rangle\in\Z \big\}=
    \spanned_\Z(\vec m_{k+1},\vec m_{k+2})$;

    $3)$ for every $\pmb\alpha\in\Omega_{k+1}$ and every
    $\vec m\in\Z^2\backslash\spanned_\Z(\vec m_{k+1},\vec m_{k+2})$, such that
    $|\vec m|<|\vec m_{k+2}|$, one has
    \[ \|\langle\pmb\alpha,\vec m_{k+1}\rangle\|<\|\langle\pmb\alpha,\vec m\rangle\|; \]

    $4)$ for every $\pmb\alpha\in\Omega_{k+1}$ and every
    $\vec m\in\Z^2\backslash\{\vec 0,\pm\vec m_k\}$, such that
    $|\vec m|<|\vec m_{k+1}|$, one has
    \[ \|\langle\pmb\alpha,\vec m_k\rangle\|<\|\langle\pmb\alpha,\vec m\rangle\|; \]

    $5)$ $(2\gamma\psi_{k+1})^{-1/2}\leq\dfrac{|\vec m_{k+2}|}{|\vec m_{k+1}|}<
    (\gamma\psi_{k+1})^{-1/2}$.

    $6)$ $\psi_k^{-1}\leq\dfrac{|\det(\vec m_{k+1},\vec m_{k+2})|}{|\vec m_k|^2}<
    \psi_k^{-1}+3\gamma$.
  \end{lemma}

  \begin{proof}
    Denote by $\tilde\Omega_k$ the ball, of which $\Omega_k$ is a half, and by $\tilde\Omega_{k+1}$
    the corresponding ball for $\Omega_{k+1}$, which is to be constructed.

    Define $\delta_k$ as follows. Set $\delta_k=1$ if for every $\vec x\in\Omega_k$
    \[ \langle\vec x,\vec m_k\rangle\leq \langle\pmb\alpha_k,\vec m_k\rangle, \]
    and set $\delta_k=-1$ if for every $\vec x\in\Omega_k$
    \[ \langle\vec x,\vec m_k\rangle\geq \langle\pmb\alpha_k,\vec m_k\rangle. \]
    Since $\ell_k$ supports $\Omega_k$, $\delta_k$ is defined correctly. Hence
    \[ \delta_k\langle\vec x,\vec m_k\rangle\leq\delta_k\langle\pmb\alpha_k,\vec m_k\rangle \]
    for every $\vec x\in\Omega_k$.

    There is exactly one integer point $\vec w$ in the parallelogram spanned by $\vec m_k$ and
    $\vec m_{k+1}$, such that the fractional part of $\langle\pmb\alpha_k,\vec w\rangle$ is minimal
    and positive, i.e.
    \[ \{\langle\pmb\alpha_k,\vec w\rangle\}=|\det(\vec m_k,\vec m_{k+1})|^{-1}. \]
    Denote by $\vec m_{k+1}^\bot$ the integer point satisfying the conditions
    $\langle\vec m_{k+1}^\bot,\vec m_{k+1}\rangle=0$, $|\vec m_{k+1}^\bot|=|\vec m_{k+1}|$ and
    $\langle\vec m_{k+1}^\bot,\vec m_k\rangle>0$. Consider the point
    \[ \vec v=
       \left(\psi_k^{-1}\frac{|\vec m_k|^2}{|\vec m_{k+1}|^2}\right)\delta_k\vec m_{k+1}^\bot-
       \sqrt{(2\gamma\psi_{k+1})^{-1}-
       \left(\psi_k^{-1}\frac{|\vec m_k|^2}{|\vec m_{k+1}|^2}\right)^2}
       \vec m_{k+1}. \]
    The subset of the affine lattice
    \[ \vec w+\spanned_\Z(\vec m_k,\vec m_{k+1}) \]
    consisting of points $\vec x$, such that the quantity
    $\langle\vec x-\vec v,\delta_k\vec m_{k+1}^\bot\rangle$ is minimal and non-negative, lies
    on a line parallel to $\vec m_{k+1}$. Define $\vec m_{k+2}$ to be the point of this set, such
    that the quantity $\langle\vec m_{k+2}-\vec v,-\vec m_{k+1}\rangle$ is minimal and
    non-negative. Notice that, due to the definition of the point $\vec m_{k+1}^\bot$, the sign of
    the coefficient $\lambda_1$ in the decomposition
    $\vec m_{k+2}=\lambda_1\vec m_k+\lambda_2\vec m_{k+1}$ is equal to that of $\delta_k$, i.e.
    \begin{equation} \label{eq:sign_of_lambda}
      \frac{\lambda_1}{|\lambda_1|}=\delta_k.
    \end{equation}
    We shall use this fact when defining $\Omega_{k+1}$. But now let us turn to the statements $5)$
    and $6)$, for their proof involves neither $\Omega_{k+1}$, nor $\alpha_{k+1}$.

    With $\vec m_{k+2}$ chosen as above the statement $5)$ follows from the inequalities
    \[ |\vec v|^2\leq|\vec m_{k+2}|^2<|\vec v|^2+|\vec m_{k+1}|^2+
       \left(\frac{\langle\vec m_k,\vec m_{k+1}^\bot\rangle}{|\vec m_{k+1}^\bot|}\right)^2<
       |\vec v|^2+\frac{|\vec m_{k+1}|^2}{2\gamma\psi_{k+1}}\,, \]
    the latter being a consequence of the relation
    \[ \frac{\langle\vec m_k,\vec m_{k+1}^\bot\rangle}{|\vec m_{k+1}^\bot|^2}=
       \frac{|\det(\vec m_k,\vec m_{k+1})|}{|\vec m_{k+1}|^2}<6\gamma^2\psi_k \]
    and the condition $\psi_k,\psi_{k+1}\leq(9\gamma)^{-1}$.

    As for the statement $6)$, it follows from the inequalities
    \[ \begin{split}
         \psi_k^{-1}\frac{|\vec m_k|^2}{|\vec m_{k+1}|^2}\leq
         \frac{\langle\vec m_{k+2},\vec m_{k+1}^\bot\rangle}{|\vec m_{k+1}^\bot|^2}< &\
         \psi_k^{-1}\frac{|\vec m_k|^2}{|\vec m_{k+1}|^2}+
         \frac{|\det(\vec m_k,\vec m_{k+1}|}{|\vec m_{k+1}|^2}\leq \\
         \leq &\ \psi_k^{-1}\frac{|\vec m_k|^2}{|\vec m_{k+1}|^2}+
         \frac{3\gamma|\vec m_k|^2}{|\vec m_{k+1}|^2}\leq
         (\psi_k^{-1}+3\gamma)\frac{|\vec m_k|^2}{|\vec m_{k+1}|^2}\,.
       \end{split} \]
    Indeed, taking into account that
    \[ \frac{|\det(\vec m_{k+1},\vec m_{k+2})|}{|\vec m_{k+1}|^2}=
       \frac{\langle\vec m_{k+2},\vec m_{k+1}^\bot\rangle}{|\vec m_{k+1}^\bot|^2} \]
    we get
    \[ \psi_k^{-1}\leq\dfrac{|\det(\vec m_{k+1},\vec m_{k+2})|}{|\vec m_k|^2}<
       \psi_k^{-1}+3\gamma. \]

    Now let us define $\pmb\alpha_{k+1}$ and $\Omega_{k+1}$ and prove the rest of the statements.
    Let us define $\pmb\alpha_{k+1}$ by the equalities
    \begin{equation} \label{eq:next_alpha_def_1}
      \langle\pmb\alpha_{k+1},\vec m_{k+1}\rangle=\langle\pmb\alpha_k,\vec m_{k+1}\rangle,
    \end{equation}
    \begin{equation} \label{eq:next_alpha_def_2}
      \langle\pmb\alpha_{k+1},\vec m_{k+2}\rangle=[\langle\pmb\alpha_k,\vec m_{k+2}\rangle].
    \end{equation}
    Note that, due to Lemma \ref{l:triple_switch}, for $\vec m_{k+2}$ and $\pmb\alpha_{k+1}$ thus
    chosen the statement $2)$ of the Lemma holds. It also follows from Lemma \ref{l:triple_switch}
    that the distance from $\pmb\alpha_{k+1}$ to $\ell_k$ is equal to
    \[ (|\vec m_k||\det(\vec m_{k+1},\vec m_{k+2})|)^{-1}, \]
    which in its turn implies that
    \[ |\pmb\alpha_{k+1}-\pmb\alpha_k|=
       \frac{|\vec m_{k+1}|}{|\det(\vec m_k,\vec m_{k+1})\det(\vec m_{k+1},\vec m_{k+2})|}\,. \]
    Hence
    \[ \begin{split}
         |\pmb\alpha_{k+1}-\pmb\alpha_k|+R_{k+1}=
         \frac{2|\vec m_{k+1}|^2}{|\det(\vec m_{k+1},\vec m_{k+2})|}R_k+
         \frac{|\det(\vec m_k,\vec m_{k+1})|}{|\det(\vec m_{k+1},\vec m_{k+2})|}
         \frac{|\vec m_{k+1}|}{|\vec m_{k+2}|}R_k< & \\ <
         \frac{2}{\gamma}R_k+3\gamma\psi_k\sqrt{2\gamma\psi_{k+1}}R_k< &\ R_k.
       \end{split} \]
    Here we have made use of the statements $5)$ and $6)$ we have already proved, the assumption
    $4)$, the condition $\psi_k,\psi_{k+1}\leq(9\gamma)^{-1}$ and the definition of $\gamma$. Thus,
    \[ \tilde\Omega_{k+1}\subset\tilde\Omega_k. \]
    More than that, $\tilde\Omega_{k+1}$ is contained either in $\Omega_k$, or in
    $\tilde\Omega_k\backslash\Omega_k$, since
    \[ R_{k+1}<(|\vec m_k||\det(\vec m_{k+1},\vec m_{k+2})|)^{-1} \]
    and the righthand side of this inequality, as we have already noticed before, is the distance
    from $\pmb\alpha_{k+1}$ to $\ell_k$.

    Let us prove that $\tilde\Omega_{k+1}\subset\Omega_k$. By \eqref{eq:next_alpha_def_2},
    \[ \langle\pmb\alpha_{k+1},\vec m_{k+2}\rangle<\langle\pmb\alpha_k,\vec m_{k+2}\rangle, \]
    so, if $\vec m_{k+2}=\lambda_1\vec m_k+\lambda_2\vec m_{k+1}$, then, due to
    \eqref{eq:next_alpha_def_1},
    \[ \langle\pmb\alpha_{k+1},\lambda_1\vec m_k\rangle<
       \langle\pmb\alpha_k,\lambda_1\vec m_k\rangle, \]
    which, in view of \eqref{eq:sign_of_lambda}, implies that
    \[ \delta_k\langle\pmb\alpha_{k+1},\vec m_k\rangle<
       \delta_k\langle\pmb\alpha_k,\vec m_k\rangle. \]
    This shows that
    \[ \tilde\Omega_{k+1}\subset\Omega_k. \]

    To define $\Omega_{k+1}$ it remains to choose between the two parts of $\tilde\Omega_{k+1}$
    separated by the line $\ell_{k+1}$. Corollary \ref{cor:distance_divisibility} together with the
    definition of $R_{k+1}$ implies that the statement $3)$ of the Lemma holds for every
    $\pmb\alpha\in\tilde\Omega_{k+1}$, so we have to specify one of the two halves of
    $\tilde\Omega_{k+1}$ only to provide the statement $4)$.

    Between the described two halves of $\tilde\Omega_{k+1}$ let us choose to be $\Omega_{k+1}$ the
    one that is closest to $\ell_k$.

    Let us prove the statement $4)$. Notice first that for each $\pmb\alpha\in\Omega_k$ the linear
    form $\langle\pmb\alpha,\cdot\rangle$ does not attain integer values at any point of the set
    $\Z^2\backslash\spanned_\Z(\vec m_k,\vec m_{k+1})$ with absolute value not exceeding
    $|\vec m_{k+1}|$. At the same time for every
    $\vec m\in\pm\vec w+\spanned_\Z(\vec m_k,\vec m_{k+1})$ we have
    \[ \|\langle\pmb\alpha_k,\vec m\rangle\|=|\det(\vec m_k,\vec m_{k+1})|^{-1}. \]
    Hence, taking into account the assumptions $4)$ and $5)$, we see that for every
    $\pmb\alpha\in\Omega_k$ and every $\vec m\in\spanned_\Z(\vec m_k,\vec m_{k+1})$, such that
    $|\vec m|\leq|\vec m_{k+1}|$, we have
    \[ \big|\langle\pmb\alpha,\vec m\rangle-\langle\pmb\alpha_k,\vec m\rangle\big|<1/2. \]
    This means that for any $\pmb\alpha\in\Omega_k$, non-collinear with $\pmb\alpha_k$, and any two
    points $\vec m',\vec m''\in\spanned_\Z(\vec m_k,\vec m_{k+1})$, such that
    $|\vec m'|,|\vec m''|\leq|\vec m_{k+1}|$, the quotient
    \[ \|\langle\pmb\alpha,\vec m'\rangle\|\Big/\|\langle\pmb\alpha,\vec m''\rangle\| \]
    is equal to the quotient of distances from the points $\vec m'$ and $\vec m''$ to the line
    \begin{equation} \label{eq:line}
      \big\{ \vec x\in\R^2\, \big|\,
      \langle\pmb\alpha,\vec x\rangle=\langle\pmb\alpha_k,\vec x\rangle \big\}.
    \end{equation}
    Due to the choice of $\Omega_{k+1}$ and the assumption $1)$, for every
    $\pmb\alpha\in\Omega_{k+1}$ the points $\pm\vec m_k$ are closer to the line \eqref{eq:line}
    than any other non-zero point of the set $\spanned_\Z(\vec m_k,\vec m_{k+1})$ with the absolute
    value less than $|\vec m_{k+1}|$. This implies the statement $4)$.
  \end{proof}

  Let us describe now the base of induction. Set $\psi_1=\psi(1)$ and
  \[ \vec m_1=(1,0),\quad
     \vec m_2=\Big(\Big\lceil\sqrt{(2\gamma\psi_1)^{-1}-\gamma^2}\Big\rceil,3\Big),\quad
     \pmb\alpha_1=\big(0,3^{-1}\big). \]
  It is easily verified that the assumptions $1)$, $2)$, $4)$, $5)$ of Lemma \ref{l:induction_step}
  are satisfied for these points. Setting
  \[ R_1=(2|\vec m_2||\det(\vec m_1,\vec m_2)|)^{-1}, \]
  choosing as $\Omega_1$ any of the two corresponding half-balls and taking into account Corollary
  \ref{cor:distance_divisibility}, we see that all the assumptions of Lemma
  \ref{l:induction_step} are fulfilled. This gives the induction base.

  Setting $\psi_k=\psi(|\vec m_k|)$ and applying Lemma \ref{l:induction_step}
  consequently for $k=2,3,4,\ldots$\ . We can do so since the assumption $5)$ together with the
  statement $6)$ of Lemma \ref{l:induction_step} for a fixed $k\in\Z_+$ imply the
  assumption  $4)$ with $k$ substituted by $k+1$. Thus we get a sequence
  $\{\Omega_k\}_{k=1}^\infty$ of embedded half-balls with a common point $\pmb\alpha$ and a
  sequence $\{\vec m_k\}_{k=1}^\infty$. It follows from the statement $4)$ of Lemma
  \ref{l:induction_step} that for each $k\in\Z_+$ the pair $\pm\vec m_k$ is the
  $k$-th pair of best approximations for $\langle\pmb\alpha,\cdot\rangle$, whereas due to the
  statements $5)$ and $6)$ we have for each $k\in\Z_+$ the inequalities
    \[ \psi_k^{-1}\leq\dfrac{|\det(\vec m_{k+1},\vec m_{k+2})|}{|\vec m_k|^2}<
       \psi_k^{-1}+3\gamma<(\psi_k-3\gamma\psi_k^2)^{-1} \]
  and
  \[ R_{k+1}|\vec m_k|^3\leq
     \Big(2(2\gamma\psi_{k+1})^{-1/2}(2\gamma\psi_k)^{-1/2}\psi_k^{-1}\Big)^{-1}\leq
     \gamma\psi_k^2. \]

  Hence for all $\pmb\alpha\in\Omega_{k+1}$
  \[ \psi_k-4\gamma\psi_k^2<
     \|\langle\pmb\alpha,\vec m_k\rangle\|\cdot|\vec m_k|^2\leq\psi_k+\gamma\psi_k^2, \]
  which proves Theorem \ref{t:main}.

  \begin{remark}
    In a similar way we can apply Lemma \ref{l:induction_step} to prove a bit different fact.
    Namely, within the assumptions of Theorem \ref{t:main} we can prove that there are continuously
    many forms $\langle\pmb\alpha,\cdot\rangle$, such that their best approximations $\vec m_k$
    satisfy the condition
    \[ \psi(k)-4\gamma\psi(k)^2<\|\langle\pmb\alpha,\vec m_k\rangle\|\cdot|\vec m_k|^2\leq
       \psi(k)+\gamma\psi(k)^2. \]
    To this effect we just have to put $\psi_k=\psi(k)$ and repeat the arguments from the proof of
    Theorem \ref{t:main}.
  \end{remark}

  \begin{remark}
    It is not difficult to improve the constants in Theorem \ref{t:main}. However, the wish to make
    better constants, if realized, would make the proof look more cumbersome and blur the main
    ideas.
  \end{remark}

\vskip1.5cm

\noindent
Oleg N. {\sc German} \\
Moscow Lomonosov State University \\
Vorobiovy Gory, GSP--1 \\
119991 Moscow, RUSSIA \\
\emph{E-mail}: {\fontfamily{cmtt}\selectfont german@mech.math.msu.su, german.oleg@gmail.com}

\vskip1.0cm

\noindent Nikolay  G. {\sc Moshchevitin} \\
Moscow Lomonosov State University \\
Vorobiovy Gory, GSP--1 \\
119991 Moscow, RUSSIA \\
\emph{E-mail}: {\fontfamily{cmtt}\selectfont moshchevitin@mech.math.msu.su,
moshchevitin@rambler.ru}

\end{document}